\documentclass[11pt]{amsart}
 
\usepackage[utf8]{inputenc}
\usepackage[T1]{fontenc}
\usepackage[english]{babel} 
\usepackage{amssymb}
\usepackage{mathrsfs} 
\usepackage{paralist}
\usepackage{multirow}
\usepackage{booktabs,multicol,array,float}
\usepackage[colorlinks=true]{hyperref} 
\usepackage{csquotes} 

\makeatletter
\renewcommand*\env@matrix[1][*\c@MaxMatrixCols c]{%
  \hskip -\arraycolsep
  \let\@ifnextchar\new@ifnextchar
  \array{#1}}
\makeatother

\newenvironment{smallpmatrix}{\left(\begin{smallmatrix}}{\end{smallmatrix}\right)}
  
\usepackage{tikz-cd}

\usepackage[backend=biber,style=alphabetic,maxbibnames=4,doi=false,isbn=false,url=false]{biblatex}
\bibliography{cpequiv}

\theoremstyle{plain}
\newtheorem{thm}{Theorem}[section]
\newtheorem{prp}[thm]{Proposition}
\newtheorem{cor}[thm]{Corollary}
\newtheorem{lem}[thm]{Lemma}
\theoremstyle{definition}
\newtheorem{dfn}[thm]{Definition}
\newtheorem{ntn}[thm]{Notation}
\theoremstyle{remark}
\newtheorem{rmk}[thm]{Remark}
\newtheorem{exa}[thm]{Example}

\newcommand{\CC}{\mathbb{C}}
\newcommand{\QQ}{\mathbb{Q}}
\newcommand{\RR}{\mathbb{R}}
\newcommand{\ZZ}{\mathbb{Z}}
\newcommand{\NN}{\mathbb{N}}

\newcommand{\KK}{\mathbb{K}}

\newcommand{\B}{\mathcal{B}}

\newcommand{\T}{\mathcal{T}}

\newcommand{\M}{\mathsf{M}}

\newcommand{\bolda}{{\boldsymbol a}}

\newcommand{\set}[1]{{\left\{#1\right\}}}
\newcommand{\abs}[1]{{\left\vert#1\right\vert}}
\newcommand{\ideal}[1]{{\left\langle#1\right\rangle}}
\newcommand{\tup}[1]{{\left(#1\right)}}

\newcommand{\onto}{\twoheadrightarrow}

\newcommand{\xmid}{\;\middle|\;}

\DeclareMathOperator{\ch}{ch}
\DeclareMathOperator{\GL}{GL}
\DeclareMathOperator{\Hom}{Hom}

\DeclareMathOperator{\Sym}{Sym}

\title[Configuration polynomials]{Configuration polynomials\\under contact equivalence}

\author[G.~Denham]{Graham Denham}
\address{\linebreak
Graham Denham\\
Department of Mathematics, University of Western Ontario\\ 
London, Ontario, Canada N6A 5B7
}
\email{\href{gdenham@uwo.ca}{gdenham@uwo.ca}}

\author[D.~Pol]{Delphine Pol}
\address{ \linebreak
Delphine Pol\\
Department of Mathematics, TU Kaiserslautern\\
67663 Kaiserslautern\\
Germany
}
\email{\href{pol@mathematik.uni-kl.de}{pol@mathematik.uni-kl.de}}

\author[M.~Schulze]{Mathias Schulze}
\address{ \linebreak
Mathias Schulze\\
Department of Mathematics, TU Kaiserslautern\\
67663 Kaiserslautern\\
Germany
}
\email{\href{mschulze@mathematik.uni-kl.de}{mschulze@mathematik.uni-kl.de}}
 
\author[U.~Walther]{Uli Walther}
\address{\linebreak 
Uli Walther\\
Department of Mathematics, Purdue University\\
West Lafayette, IN 47907, USA
}
\email{\href{walther@math.purdue.edu}{walther@math.purdue.edu}}
 
\subjclass[2010]{Primary 14N20; Secondary 05C31, 14M12, 81Q30}

\keywords{Configuration, matroid, contact equivalence, Feynman, Kirchhoff, Symanzik}

\thanks{GD supported by NSERC of Canada. 
DP supported by a Humboldt Research Fellowship for Postdoctoral Researchers.
UW supported in part by Simons Foundation Collaboration Grant for Mathematicians \#580839.}

\numberwithin{equation}{section} 

\begin{document}

\begin{abstract}
Configuration polynomials generalize the classical Kirchhoff polynomial defined by a graph.
Their study sheds light on certain polynomials appearing in Feynman integrands.
Contact equivalence provides a way to study the associated configuration hypersurface. 
In the contact equivalence class of any configuration polynomial we
identify a polynomial with minimal number of variables; it is a configuration polynomial.
This minimal number is bounded by $r+1\choose 2$, where $r$ is the rank of the underlying matroid.
We show that the number of equivalence classes is finite exactly up to rank $3$ and list explicit normal forms for these classes.
\end{abstract}

\maketitle
\tableofcontents

\section{Introduction}

The Matrix-Tree Theorem is a classical result in algebraic graph theory. 
It was found by German physicist Gustav Kirchhoff in the mid-19th century in the study of electrical circuits.
It states that the number of spanning trees of a connected undirected graph $G$ with edge set $E$ agrees with any principal submaximal minor of its Laplacian.
Putting weights on the edges $e\in E$ of $G$ and considering them as variables $x_e$ yields the \emph{Kirchhoff polynomial} 
\[
\psi_G=\sum_{T\in\T_G}x^T,
\]
where $\T_G$ is the set of all spanning trees of $G$, and $x^T=\prod_{e\in T}x_e$.

Kirchhoff polynomials are a crucial ingredient of the theory of Feynman integrals (see, for example, \cite{Alu14,BBKP19,BS12,BSY14} and the literature trees in these works). 
In short, the Kirchhoff polynomial of a graph appears in the denominator of the Feynman integral attached to the particle scattering encoded by the dual graph via Feynman's rule. 
In certain cases, the integrand is just a power of the Kirchhoff polynomial, but in general there is also another component, a \emph{second Symanzik polynomial}.
In this way, singularities of Kirchhoff polynomials influence the behavior of the corresponding Feynman integral.

Considered as functions over $\KK=\CC$, Kirchhoff polynomials are never zero if all variables take values in a common open half-plane (defined by positivity of a non-trivial $\RR$-linear form).
Because Kirchhoff polynomials are homogeneous, this property is independent of the choice of half-plane.
For the right half-plane (with positive real part) it is referred to as the \emph{(Hurwitz) half-plane property}; the upper half-plane (with positive imaginary part) defines the class of \emph{stable polynomials}.
Generalizing beyond graphs, any matroid $\M$ with set of bases $\B_\M$ defines a \emph{matroid basis polynomial} $\psi_\M=\sum_{B\in\B_\M}x^B$.
In this way, $\psi_G=\psi_{\M_G}$ depends only on the \emph{graphic matroid} $\M_G$ on $E$ with set of bases $\T_G$.
Conditions for the half-plane property of $\psi_\M$ in terms of $\M$ were formulated by Choe, Oxley, Sokal and Wagner (see \cite[92]{COSW04}).
They consider general polynomials $\psi_{\M,\bolda}=\sum_{B\in\B_\M}a_Bx^B$ with matroid support and arbitrary coefficients $\bolda=(a_B)_{B\in\B_\M}$.
The question whether the half-plane property of $\psi_{\M,\bolda}$ for some coefficients $\bolda$ descends to $\psi_\M$ is studied for example by Br\"{a}nd\'{e}n and Gonz\'{a}lez D'Le\'{o}n (see \cite[Thm.~2.3]{BG10}), while Amini and Br\"{a}nd\'{e}n (see \cite{AB18}) consider interactions of the half-plane property, representability and the Lax conjecture.

Recently, Br\"and\'en and Huh introduced the class of \emph{Lorentzian polynomials} (see \cite{BH20}), which are defined by induction over the degree using partial derivatives, starting from quadratic forms satisfying a signature condition.
Stable polynomials are Lorentzian.  These polynomials have interesting negative dependence properties and close relations with matroids.
For example, if a multiaffine polynomial (that is, a polynomial supported on squarefree monomials) is Lorentzian, then it has the form $\psi_{\M,\bolda}$ for some matroid $\M$ and positive coefficients $\bolda$.  

By the Matrix-Tree Theorem, however, the coefficients $1$ of the Kirchhoff polynomial arise in a particular way:
Pick any orientation on $G$ and let $A$ be an incidence matrix with one row deleted.
Then $\psi_G=\det(AX A^\intercal)$, where $X$ the diagonal matrix of variables $x_e$ for all $e\in E$.
In more intrinsic terms, this is the determinant of the generic diagonal bilinear form, restricted to the span $W_G\subseteq\ZZ^E$ of all incidence vectors. 
Bloch, Esnault and Kreimer took this point of view for any linear subspace $W\subseteq\KK^E$ over a field $\KK$ (see \cite{BEK06,Pat10}).
With respect to the basis of $\KK^E$ this is a linear realization of a matroid $\M$, or a \emph{configuration}.
The dimension $\dim W$ equals the rank of $\M$, which we refer to as the \emph{rank} of the configuration $W$ (see Definition~\ref{3}).
The generic diagonal bilinear form on $\KK^E$ restricts to a \emph{configuration form} $Q_W$ on $W$.
Its determinant $\psi_W=\det(Q_W)$ is the \emph{configuration polynomial} associated with $W$, a homogeneous polynomial of degree $\dim W$ in variables $x_e$ for all $e\in E$ (see Definitions~\ref{4} and \ref{15}).
Configuration polynomials over $\KK=\CC$ are stable, by a result of Borcea and Br\"and\'en (see \cite[Prop.\ 2.4]{BB08}).
Notably the above mentioned second Symanzik polynomial is a configuration polynomial, but not a Kirchhoff polynomial.

\smallskip 

The configuration point of view has recently led to new insights on the affine and projective hypersurfaces defined by Kirchhoff polynomials (see \cite{DSW21,DPSW20}).
At present, the understanding of all the details of the singularity structure, as well as a satisfactory general treatment of Feynman integrals, is highly incomplete. 
There is some evidence that this is due to built-in complications coming from complexity issues (see \cite{BB03}). 
A natural problem is then to determine to what extent the formula for a configuration hypersurface is the most efficient way to encode the geometry: given a configuration polynomial, can it be rewritten
in fewer variables, and can this even be done as via another configuration?


In this article, we elaborate on this idea by studying configurations through the lens of \emph{(linear) contact equivalence} of their corresponding polynomials.
This is the equivalence relation on polynomials induced by permitting coordinate changes on the source and target of the polynomial (see Definition~\ref{8}).
Polynomials in the same equivalence class define the same affine hypersurfaces, up to a product with an affine space. 
While this approach is very natural from a geometric point of view, forgetting
the matroid structure under the equivalence makes it difficult to navigate, and provides certain surprises discussed below. 

The main vehicle of our investigations is that any matrix representation of $Q_W$ consists of \emph{Hadamard products} $v\star w$ of vectors $v,w\in W$, defined with respect to a basis of $\KK^E$ (see Notation~\ref{18}). 
After some preliminary discussion in earlier sections, we focus in \S\ref{67} on the problem of finding \enquote{small} representatives within the contact equivalence class of a given configuration. 
This requires us to look in detail at the structure of the higher \emph{Hadamard powers} $W^{\star s}$ of $W$ (see \S\ref{64}). 
While such Hadamard powers do usually not form chains with increasing $s$, they nonetheless have some monotonicity properties with regard to suitable restrictions to subsets of $E$ (see Lemma~\ref{1}). 
We use this to minimize the number of variables of configuration polynomials under contact equivalence (see Proposition~\ref{2}).
As a result we obtain


\begin{thm}\label{5}
Let $W\subseteq\KK^E$ be a configuration over a field $\KK$ of characteristic $\ch\KK=0$, or $\ch\KK>\dim W$.
Let $\nu(W)$ be the minimal number of variables appearing in any polynomial contact equivalent to
$\psi_W$. 
Then 
\[
\nu_W\le{\dim W+1\choose 2}.
\]
This minimum is realized within the set of configuration polynomials: there is a configuration $W'\subseteq \KK^{\nu(W)}$, constructed from $W$ by a suitable matroid restriction, with $\psi_W$ and $\psi_{W'}$ contact equivalent.
\end{thm}


In \S\ref{68}, \S\ref{69} and \S\ref{71}, we then consider the classification problem of determining all contact equivalence classes for configurations $W$ of a given rank, and we prove the following


\begin{thm}\label{6}
For configurations of rank up to $3$, there are only finitely many contact equivalence classes. 
For each rank at least $4$, there is an infinite family of pairwise inequivalent configurations over $\KK=\QQ$.
\end{thm}


More precisely, we identify for $\dim W\le 3$ all contact equivalence classes and write down a normal form for each class (see Table~\ref{87}). 
This list is made of all possible products of generic determinants in up to $6$ variables together with 
\[
\det
\begin{pmatrix}
y_1 & y_4 & y_5\\
y_4 & y_2 & 0\\
y_5   & 0 & y_3
\end{pmatrix}\quad\text{and}\quad
\det
\begin{pmatrix}
y_1 & y_4 & y_4+y_5 \\
y_4 & y_2 & y_5 \\
y_4+y_5 & y_5 & y_3
\end{pmatrix}.
\]
For $\dim W=4$, already for $\abs E=6$ variables, we exhibit an infinite family of contact equivalence classes of configurations (see Proposition~\ref{53}). 


Our computations show that the contact equivalence class of a configuration neither determines nor is determined by the underlying matroid. 
Thus, one is prompted to wonder what characteristics of the graph/matroid of a Kirchhoff/configuration polynomial determine its complexity. 
We hope that our investigations here will help to shed light on this problem.

\subsection*{Acknowledgments}

We gratefully acknowledge support by the Bernoulli Center at EPFL during a \enquote{Bernoulli Brainstorm} in February 2019, and by the Centro de Giorgi in Pisa during a \enquote{Research in Pairs} in February 2020. 
We also thank June Huh for helpful comments.

\section{Configuration forms and polynomials}

Let $\KK$ be a field.
We denote the dual of a $\KK$-vector space $W$ by
\[
W^\vee:=\Hom_{\KK}(W,\KK).
\]

Let $E$ be a finite set. 
Whenever convenient, we order $E$ and identify 
\[
E=\set{e_1,\dots,e_n}=\set{1,\dots,n}.
\]
We identify $E$ with the canonical basis of the based $\KK$-vector space
\[
\KK^E:=\bigoplus_{e\in E}\KK\cdot e.
\]
We denote by $E^\vee=(e^\vee)_{e\in E}$ the dual basis of 
\[
(\KK^E)^\vee=\KK^{E^\vee}.
\]
We write $x_e:=e^\vee$ to emphasize that $x:=(x_e)_{e\in E}$ is a coordinate system on $\KK^E$.
For $F\subseteq E$ we denote by
\[
x^F:=\prod_{f\in F}x_f
\]
the corresponding monomial.
For $w\in\KK^E$ and $e\in E$, denote by $w_e:=e^\vee(w)$ the $e$-component of $w$.


\begin{dfn}\label{3}
Let $E$ be a finite set. 
A \emph{configuration} over $\KK$ is a $\KK$-vector space $W\subseteq \KK^E$. 
It gives rise to an \emph{associated matroid} $\M=\M_W$ with rank function $S\mapsto\dim_\KK\ideal{S^\vee|_W}$ and set of bases $\B_\M$.
We refer to its rank
\[
r_W:=\dim_\KK W
\]
as the \emph{rank} of the configuration.
\emph{Equivalent} configurations obtained by rescaling $E$ or by applying a field automorphism have the same associated matroid.
\end{dfn}


\begin{ntn}\label{18}
We denote the \emph{Hadamard product} of $u,v\in\KK^E$ by
\[
u\star v:=\sum_{e\in E} u_e\cdot v_e\cdot e\in\KK^E.
\]
We suppress the dependency on $E$ in this notation.
We abbreviate
\[
u^{\star s}:=\underbrace{u\star\cdots\star u}_{s}.
\]
\end{ntn}


\begin{dfn}[{\cite[Rem.~3.21, Def.~3.20]{DSW21},\cite[\S2.2]{Oxl11}}]\label{4}
Denote by $\mu_\KK$ the multiplication map of $\KK$.
Let $W\subseteq \KK^E$ be a configuration of rank $r=r_W$.
The associated \emph{configuration form} is 
\[
Q_W=\sum_{e\in E}x_e\cdot\mu_\KK\circ \left(e^\vee\times e^\vee\right)\colon W\times W\to\ideal{x}_\KK.
\]

A choice of (ordered) basis $w=(w^1,\dots,w^r)$ of $W\subseteq \KK^E$ together with an ordering of $E$ is equivalent to the choice of a \emph{configuration matrix} $A=(w^i_j)_{i,j}\in\KK^{r\times n}$ with row span $\ideal A$ equal to $W$.
With respect to these choices, $Q_W$ is represented by the $r\times r$ matrix
\[
Q_w:=Q_A:=(\ideal{x,w^i\star w^j})_{i,j}=\tup{\sum_{e\in E} x_e\cdot w_e^i\cdot w_e^j}_{i,j}.
\]
Different choices of bases $w,w'$ and orderings (or, equivalently, of configuration matrices) yield conjugate matrix representatives for $Q_W$.

Judicious choices of the basis and the orderings lead to a \emph{normalized} configuration matrix $A=\begin{pmatrix}I_r & A'\end{pmatrix}$, where $I_r$ is the $r\times r$ unit matrix.
\end{dfn}


\begin{rmk}\label{43}
For fixed $e\in E$, $(w_e^i\cdot w_e^j)_{i\le j}$ is the image of $(w_e^i)_i$ under the second Veronese map $\KK^r\to \KK^{r\choose 2}$.
Thus, $Q_w$ determines the vectors $(w_e^i)_i$ up to a common sign.
In particular, $Q_W$ determines the configuration $W$ up to equivalence.
\end{rmk}


\begin{dfn}[{\cite[Def.~3.2, Rem.~3.3, Lem.~3.23]{DSW21}}]\label{15}
Let $W\subseteq \KK^E$ be a configuration. 
If $A$ is a configuration matrix for $W$ with corresponding basis $w$, then the associated \emph{configuration polynomial} is defined by
\[
\psi_W:=\psi_w:=\psi_A:=\det(Q_A)\in\KK[x].
\]
It is determined by $W$ up to a square factor in $\KK^*$.
One has the alternative description
\[
\psi_A=\sum_{B\in\B_\M}\det(\KK^B\overset{w}{\to}W\onto\KK^B)^2\cdot x^B,
\]
using the ordering corresponding to $A$ on every basis $B\subseteq E$. 

The \emph{matroid (basis) polynomial} 
\[
\psi_\M=\sum_{B\in\B_\M} x^B\in\ZZ[x]
\]
of $\M=\M_W$ has the same monomial support as $\psi_W$ but the two can be significantly different (see \cite[Ex.~5.2]{DSW21}).
\end{dfn}


\begin{rmk}
If $G=(V,E)$ is a graph and $W\subseteq\KK^E$ is the row span of the incidence matrix of $G$, then $\psi_W=\psi_G$ is the \emph{Kirchhoff polynomial} of $G$ (see \cite[Prop.~3.16]{DSW21}).
\end{rmk}

\section{Hadamard products of configurations}\label{64}

Let $W\subseteq\KK^E$ be a configuration of rank
\[
r=r_W=\dim_\KK W\le\abs{E}.
\]
For $s\in\NN_{\ge1}$, denote by
\[
W^{\star s}:=\underbrace{W\star\cdots\star W}_{s}:=\ideal{w^1\star\cdots\star w^s\xmid w^1,\dots,w^s\in W}\subseteq\KK^E
\]
the $s$-fold Hadamard product of $W$ and by 
\[
r_W^s:=\dim_\KK W^{\star s}\le\abs{E}
\]
its dimension.
Note that $r_W=r_W^1$
By multilinearity and symmetry of the Hadamard product, we have a surjection
\[
\Sym^s_\KK W\onto W^{\star s},\quad w^{i_1}\cdots w^{i_s}\mapsto w^{i_1}\star\cdots\star w^{i_s}.
\]
In particular, for all $s,s'\in\NN_{\geq 1}$, there is an estimate
\begin{equation}\label{7}
r_W^s\le{r_W+s-1\choose s}.
\end{equation}
and equations
\[
(\KK^E)^{\star s}=\KK^E,\quad W^{\star s}\star W^{\star s'}=W^{\star (s+s')}.
\]


\begin{exa}
Consider the non-isomorphic rank $2$ configurations in $\KK^n$
\[
W=\ideal{(1,\dots,1),(1,2,3,\dots,n)},\quad W'=\ideal{(1,0,\ldots,0),(0,1,0,\ldots,0)}.
\]
Then $r_W^s=\min\set{s,n}$ as follows from properties of Vandermonde determinants, whereas $r_{W'}^s=2$.
\end{exa}


\begin{rmk}\label{12}
Extending a configuration $W\subseteq\KK^E$ by a direct summand $\KK$ with basis $f$ yields a new configuration $W'=W\oplus\KK^\set{f}\subseteq\KK^{E\sqcup\set{f}}$ with configuration matrix $A'=\begin{pmatrix} A& 0\\0 & 1\end{pmatrix}$, $r_{W'}^s=r_W^s+1$ and $\psi_{W'}=\psi_W\cdot x_f$.
\end{rmk}


For $F\subseteq E$, denote by 
\[
\pi_F\colon\KK^E\to\KK^F
\]
the corresponding $\KK$-linear projection map.
Abbreviate
\[
w_F:=\pi_F(w),\quad W_F:=\pi_F(W).
\]
By definition, $(w^1\star\cdots\star w^s)_F=w_F^1\star\cdots\star w_F^s$ and hence
\[
(W^{\star s})_F=(W_F)^{\star s}=:W^{\star s}_F.
\]


\begin{lem}\label{1}
For every configuration $W\subseteq\KK^E$ there is a filtration
\[
F_1\subseteq\cdots\subseteq F_t\subseteq\cdots\subseteq E
\]
on $E$ such that, for all $s'\le s$ in $\NN_{\geq 1}$, there is a commutative diagram
\begin{equation}\label{1a}
\begin{tikzcd}
\KK^E\arrow[r]\arrow[r,"\pi_{F_s}",twoheadrightarrow] & \KK^{F_s} \\[-5mm]
W^{\star s'}\ar[u,"\subseteq" description, sloped]\ar[r,"\cong"'] &  W^{\star s'}_{F_s}\ar[u,"\subseteq" description, sloped]
\end{tikzcd}
\end{equation}
in which the right hand containment is an equality for $s'=s$. 
In particular, for $s'\le s$,
\begin{equation}\label{1b}
r_W^{s'}\le r_W^s.
\end{equation}
\end{lem}

\begin{proof}
Note that \eqref{1b} is a direct consequence of \eqref{1a} and the filtration property.
We will construct the filtration inductively, starting with $F_1$. 
Let $F_1$ be any subset of $E$ such that $r_{W_{F_1}}=\abs{F_1}$ (in other words, a basis for the matroid $\M_W$ represented by $W$). 
Then \eqref{1a} is clear.

Suppose that $F_1\subseteq\cdots\subseteq F_t$ have been constructed, satisfying \eqref{1a} whenever $s'\le s\le t$. 
We claim first that $W^{\star(t+1)}_{F_s}=\KK^{F_s}$ for all $1\le s\le t$.
So take a basis element $e\in F_s$. 
From the inductive hypothesis $W_{F_s}^{\star s}=\KK^{F_s}$ we obtain a $v\in W^{\star s}$ such that $v_{F_s}=e$. 
By definition of $W^{\star s}$, there must be a $u\in W$ such that $u_e=1$ as otherwise $W_e=0$.  But then $w:=u^{\star (t+1-s)}\star v\in W^{\star(t+1)}$ satisfies $w_{F_s}=e$, so that $W^{\star(t+1)}_{F_s}=\KK^{F_s}$ as claimed.

The just established equation $W^{\star(t+1)}_{F_t}=\KK^{F_t}$ says that $F_t$ is an independent set for the matroid associated to the configuration $W^{\star(t+1)}\subseteq \KK^E$. 
Extend it to a basis $F_{t+1}$. 
Then \eqref{1a} follows for $s'=s=t+1$ (including the equality of the right inclusion). 
On the other hand, for $s'\le t$, the natural composite surjection
\[
\begin{tikzcd}
W^{\star s'}\ar[r,twoheadrightarrow] &  W^{\star s'}_{F_{t+1}}\ar[r,twoheadrightarrow] & W^{\star s'}_{F_t}
\end{tikzcd}
\]
is by the inductive hypothesis an isomorphism. 
Hence each of the two arrows in the display is an isomorphism as well, proving that \eqref{1a} holds for $s'<s=t+1$.  
\end{proof}


\begin{dfn}
Let $W\subseteq\KK^E$ be a configuration.
By Lemma~\ref{1} there is a minimal index $t_W$ such that $r_W^t=r_W^{t_W}$ for all $t\geq t_W$. 
We call $t_W$ the \emph{Hadamard exponent} and $r_W^{t_W}$ the \emph{Hadamard dimension} of $W$.
\end{dfn}

\section{Linear contact equivalence}

\begin{dfn}\label{8}
We call two polynomials $\phi\in\KK[x_1,\ldots,x_m]$ and $\psi\in\KK[x_1,\ldots,x_n]$ \emph{(linearly contact) equivalent} if for some $p\geq m,n$ there exists an $\ell\in\GL_p(\KK)$ and a $\lambda\in\KK^*$ such that 
\begin{equation}\label{9}
\phi=\lambda\cdot\psi\circ\ell
\end{equation}
in $\KK[x_1,\ldots,x_p]$.
We write $\phi\simeq\psi$ in this case. 
\end{dfn}


\begin{rmk}\label{61}\
\begin{enumerate}[(a)]

\item\label{61a} If $\KK$ is a perfect field and $\psi$ is homogeneous, then one can assume $\lambda=1$ in \eqref{9} at the cost of scaling $\ell$ by $\lambda^{1/\deg(\psi)}$. 

\item\label{61b} By definition, both adding redundant variables and permuting variables yield equivalent polynomials.
In particular enumerating $E$ and considering $E\subseteq\set{1,\dots,p}$ as a subset for any $p\ge\abs E$ gives sense to equivalence of configuration polynomials $\psi_W$.

\end{enumerate}
\end{rmk}


\begin{ntn}
For a fixed field $\KK$, we set
\[
\Psi:=\set{\psi_W \xmid E\text{ finite set},\ W\subseteq\KK^E}.
\]
\end{ntn}


We aim to understand linear contact equivalence on $\Psi$.

\section{Reduction of variables modulo equivalence}\label{67}

\begin{lem}\label{40}
Let $W\subseteq \KK^E$ be a configuration.
Then there is a subset $F\subseteq E$ of size $\abs{F}=r_{W_F}^2=r_W^2$ such that $\psi_W\simeq\psi_{W_F}$.
\end{lem}

\begin{proof}
Lemma~\ref{1} with $t=2$ yields a subset $F\subseteq E$ such that
\begin{equation}\label{59}
\pi_F\vert_W\colon W\stackrel{\cong}{\longrightarrow} W_F\quad \text{ and
}\quad\pi_{F}\vert_{W^{\star 2}}\colon
W^{\star2}\stackrel{\cong}{\longrightarrow}W^{\star  2}_F=\KK^{F}.
\end{equation}
Let $\iota_F$ be the section of $\pi_F$ that factors through the inverse of $\pi_{F}\vert_{W^{\star 2}}$,
\begin{equation}\label{60}
\begin{tikzcd}[column sep=large]
\iota_F\colon\KK^F \ar[r,"(\pi_{F}\vert_{W^{\star 2}})^{-1}"] & W^{\star2}\ar[r,hookrightarrow] & \KK^E.
\end{tikzcd}
\end{equation}
Consider the $\KK$-linear isomorphism of based vector spaces
\[
q\colon\KK^E\to\KK^{E^\vee},\quad w\mapsto\sum_{e\in E}w_e\cdot x_e
\]
inducing the configuration $q(W)\subseteq\KK^{E^\vee}$. 
Set $F^\vee:=q(F)$ and $\iota_{F^\vee}:=q\circ\iota_F\circ q^{-1}$. 
Then $\pi_{F^\vee}=q\circ \pi_F\circ q^{-1}$, and \eqref{59} and \eqref{60} persist if $F$ is replaced by $F^\vee$ and $W$ by $q(W)$ throughout. 

Now choose a basis $w=(w^1,\dots,w^r)$ of $W$.
Then $w_F=(w^1_F,\dots,w^r_F)$ is a basis of $W_F$ by \eqref{59} and 
\begin{align*}
Q_W&=\tup{q(w^i\star w^j)}_{i,j}\\
&=\tup{q(w^i)\star q(w^j)}_{i,j}\\
&\overset{\eqref{60}}=\tup{\iota_{F^\vee}\circ \pi_{F^\vee}(q(w^i)\star q(w^j))}_{i,j}\\
&=\iota_{F^\vee}\tup{q(w^i)_{F^\vee}\star q(w^j)_{F^\vee}}_{i,j}\\
&=\iota_{F^\vee}\tup{q(w^i_F)\star q(w^j_F)}_{i,j}\\
&=\iota_{F^\vee}\tup{q(w_F^i\star w_F^j)}_{i,j}=\iota_{F^\vee}Q_{W_F}.
\end{align*}
Since $\iota_{F^\vee}$ is a section of $\pi_{F^\vee}$, $\psi_W\simeq\psi_{W_F}$ by taking determinants.
\end{proof}


\begin{lem}\label{41}
Let $W\subseteq\KK^E$ be a configuration.
Suppose that $\ch\KK=0$, or $\ch\KK>r_W$.
If $\psi_W\simeq\phi\in\KK[y_1,\dots,y_{n-1}]$ where $n:=\abs E$, then $\psi_W\simeq\psi_{W_{E\setminus\set{e}}}$ for some $e\in E$.
\end{lem}

\begin{proof}
Let $\ell\in\GL_p(\KK)$ and $\lambda\in\KK^*$ realize the equivalence $\phi\simeq\psi_W$, that is, $\phi=\lambda\cdot\psi_W\circ\ell$ where $E\subseteq\set{1,\dots,p}$ (see Remark~\ref{61}.\eqref{61b}). 
Consider the $\KK$-linearly independent $\KK$-linear derivations of $\KK[x_1,\dots,x_p]$
\[
\delta_i:=\ell_*(\frac{\partial}{\partial y_{n-1+i}})=\frac{\partial}{\partial y_{n-1+i}}(-\circ\ell)\circ\ell^{-1},\quad i=1,\dots,p-n+1.
\]
Since $\phi$ is independent of $y_n,\dots,y_p$, we have
\begin{equation}\label{62}
\delta_i(\psi_W)=\lambda^{-1}\cdot\frac{\partial\phi}{\partial y_{n-1+i}}\circ\ell^{-1}=0,\quad i=1,\dots,p-n+1.
\end{equation}
By suitably reordering $\set{1,\dots,p}$ we may assume that the matrix $(\delta_i(x_{j}))_{i,j\in\set{1,\dots, p-n+1}}$ is invertible.
After replacing the $\delta_i$ by suitable linear combinations, we may further assume that $\delta_i(x_j)=\delta_{i,j}$ for all $i,j\in\set{1,\dots, p-n+1}$.
Then
\begin{align*}
x_i&=x'_i,\quad i=1,\dots,p-n+1,\\
x_i&=x'_i+\sum_{j=1}^{p-n+1}\delta_j(x_i)\cdot x'_j,\quad i=p-n+2,\dots,p,
\end{align*}
defines a coordinate change such that
\begin{equation}\label{63}
\delta_j=\sum_{i=1}^p\delta_j(x_i)\frac{\partial}{\partial x_i}
=\sum_{i=1}^p\frac{\partial x_i}{\partial x'_j}\frac{\partial}{\partial x_i}=\frac{\partial}{\partial x'_j},\quad j=1,\dots,p-n+1.
\end{equation}
If $\ch\KK>0$, then $\ch\KK>r_W=\deg(\psi_W)$ by hypothesis.
By \eqref{62} and \eqref{63}, $\psi_W$ is thus independent of $x'_1,\dots,x'_{p-n+1}$.
Setting $x_i=x'_i=0$ for $i=1,\dots,p-n+1$ thus leaves $\psi_W$ unchanged and makes $x_i=x'_i$ for $i=p-n+2,\dots,p$.
It follows that
\[
\psi_W\simeq\psi_W\vert_{x'_1=\cdots=x'_{p-n+1}=0}
=\psi_W\vert_{x_1=\cdots=x_{p-n+1}=0}
=\psi_{W_{E\setminus\set{1,\dots,p-n+1}}}.
\]
Then any $e\in E\cap\set{1,\dots,p-n+1}$ satisfies the claim.
\end{proof}


\begin{prp}\label{2}
Let $W\subseteq \KK^E$ be a configuration.
Then there is a subset $F\subseteq E$ of size $\abs{F}=r_{W_F}^2\le r_W^2$ such that $\psi_W\simeq\psi_{W_F}$.
Suppose that $\ch\KK=0$, or $\ch\KK>r_W$.
Then any polynomial $\phi\simeq\psi_{W_F}$ depends on at least $\abs{F}$ variables.
In other words, among the polynomials equivalent to $\psi_W$ with minimal number of variables is the configuration polynomial $\psi_{W_F}$.
\end{prp}

\begin{proof}
By Lemma~\ref{40} there is a subset $G\subseteq E$ such that $\abs{G}=r_{W_{G}}^2= r_W^2$ and $\psi_W\simeq\psi_{W_G}$.
Note that $\abs{G}=r_{W_G}^2$ means $W_G^{\star 2}=\KK^G$ which for any subset $F\subseteq G$ implies that $W_F^{\star 2}=\KK^F$ and hence $\abs{F}=r_{W_F}^2\le r_W^2$.
Pick such an $F$ with $\psi_{W_F}\simeq\psi_{W_G}$ minimizing $\abs F$.
Note that $r_{W_F}\le r_W$.
By Lemma~\ref{41} applied to the configuration $W_F\subseteq\KK^F$, any $\phi\simeq\psi_{W_F}$ depending on fewer than $\abs{F}$ variables yields an $e\in F$ such that $\psi_{W_F}\simeq\psi_{W_{F\setminus\set e}}$, contradicting the minimality of $F$.
\end{proof}


\begin{rmk}\label{42}
By Remark~\ref{43}, $Q_W$ determines $r_W^2$.
By definition, (the equivalence class of) $\psi_W$ determines $r_W^1=r_W=\deg\psi_W$.
We do not know whether it also determines  $r_W^2$.
\end{rmk}

\section{Extremal cases of equivalence classes}\label{68}

\begin{ntn}
For $r,d\in\NN$, set 
\[
\Psi^d_r=\set{\psi_W \mid E\text{ finite set},\ W\subseteq\KK^E,\ r_W=r,\ r_W^2=d}.
\]
\end{ntn}


\begin{lem}\label{10}
Let $W\subseteq \KK^E$ be a configuration of rank $r$ with basis $(w^1,\dots,w^r)$. 
Let $G$ be the graph on the vertices $v_1,\ldots,v_r$ in which $\{v_i,v_j\}$ is an edge if and only if $w^i\star w^j\ne0$. 
Let $G^*$ be the cone graph over $G$.

If $\set{w^i\star w^j\mid i\le j, w^i\star w^j\neq 0}$ is linearly
independent, then
\[
\psi_W\simeq\psi_{G^*}
\]
is the Kirchhoff polynomial of $G^*$.
\end{lem}

\begin{proof}
See \cite[Thm.~3.2]{BB03} and its proof.
\end{proof}


\begin{prp}\label{11}
If $d=r$, then every element of $\Psi^d_r$ is equivalent to $x_1\cdots x_r$.

If $d={r+1\choose 2}$, then every element of $\Psi^d_r$ is equivalent to the elementary symmetric polynomial of degree $r$ in the variables $x_1,\ldots,x_d$.
\end{prp}

\begin{proof}
Let $W\subseteq\KK^E$ be a configuration.

First suppose that $r_W^2=r_W$.
By Lemma~\ref{40}, we may assume that $\abs{E}=r_W^2$.
Then $W=\KK^E$ and hence $\psi_W=x^E$ is the matroid polynomial of the free matroid on $r_W$ elements.

Now suppose that $r_W^2={r_W+1\choose 2}$.
Then $\set{w^i\star w^j\mid 1\le i\le j\le r}$ is linearly independent for any basis $(w^1,\ldots,w^r)$ of $W$.
By Lemma~\ref{10}, $\psi_W$ is then equivalent to the Kirchhoff polynomial of the complete graph on $r_W+1$ vertices. 
\end{proof}

\section{Finite number classes for small rank matroids}\label{69}

The purpose of this section is to give a complete classification of configuration polynomials for matroids of rank at most $3$ with respect to the equivalence relation of Definition~\ref{8}.
Due to Proposition~\ref{2}, we may assume that $\abs E=r_W^2$.

 
\begin{dfn}[{\cite[\S2.2]{Oxl11}}]\label{13}
A choice of basis $(w^1,\dots,w^r)$ of $W\subseteq \KK^E$ and order of $E$ gives rise to a \emph{configuration matrix} $A=(w^i_j)_{i,j}\in\KK^{r\times n}$, whose row span recovers $W=\ideal A$.
Up to reordering $E$ it can be assumed in \emph{normalized} form $A=(I_r|A')$ where $I_r$ is the $r\times r$ unit matrix.
\end{dfn}

 
\begin{prp}\label{16}
Let $W$ be a configuration of rank $2$.
If $r_W^2=2$, then $\psi_W\simeq x_1x_2$, otherwise, $r_W^2=3$ and $\psi_W\simeq x_1x_2-x_3^2$. 
\end{prp}

\begin{proof}
Most of this follows from the proof of Proposition~\ref{11}.
Apply $x_1\mapsto x_1+x_2$ to the Kirchhoff polynomial $x_1x_2+x_2x_3+x_3x_1$ of $K_3$; the result is $x_1^2+x_1(x_2+2x_3)+x_2x_3$.

If $\ch\KK=2$, then this is $x_1^2+x_2(x_1+x_3)$. 
If $2\in\KK$ is a unit, complete the square and scale $x_2$ by $2$ to arrive at $x_1^2-x_2^2+x_3^2$. 
In both cases the result is easily seen to be equivalent to $x_1x_2-x_3^2$.
\end{proof} 

 
\begin{prp}\label{17}
The numbers of equivalence classes for rank $3$ configurations $W$ for different values of $r_W^2$ are 
\[
|\Psi_3^3/_\simeq|=1,\quad
|\Psi_3^4/_\simeq|=2,\quad
|\Psi_3^5/_\simeq|=2,\quad
|\Psi_3^6/_\simeq|=1.
\]
Table \ref{87} lists the equivalence classes of $\psi_W$ that arise from normalized configuration matrices $A$ when $r_W=3$ and $r_W^2=\abs E$.
\end{prp}

\begin{table}[H]
\caption{Equivalence classes for rank $r_W=3$ configurations}\label{87}
\begin{tabular}{ccm{50mm}c@{}}
\toprule
$\abs E=r^2_W$ & $A$ & conditions & $\psi_W\simeq\det(-)$ \\

\midrule
$3$ &
$\begin{smallpmatrix}
1 & 0 & 0\\
0 & 1 & 0\\
0 & 0 & 1
\end{smallpmatrix}$ & None &
$\begin{smallpmatrix}
y_1 & 0 & 0\\
0 & y_2 & 0\\
0 & 0 & y_3
\end{smallpmatrix}$ \\

\midrule
\multirow{3}{*}{$4$} & 
$\begin{smallpmatrix}
1 & 0 & 0  & a_1 \\
0 & 1 & 0 & a_2\\
0 & 0 & 1 & a_3
\end{smallpmatrix}$ &
$a_i=0$ for exactly one $i$. &
$\begin{smallpmatrix}
y_1 & y_4 & 0\\
y_4 & y_2 & 0\\
0   & 0   & y_3
\end{smallpmatrix}$ \\

\cmidrule{2-4}
 & $\begin{smallpmatrix}
1 & 0 & 0 & a_1\\
0 & 1 & 0 & a_2\\
0 & 0 & 1 & a_3
\end{smallpmatrix}$ &
$a_i\ne0$ for all $i$. &
$\begin{smallpmatrix}
y_1 & y_4 & y_4\\
y_4 & y_2 & y_4 \\
y_4 & y_4 & y_3
\end{smallpmatrix}$ \\

\midrule
\multirow{3}{*}{$5$} & $\begin{smallpmatrix}
1 & 0 & 0  & a_{1,1} & a_{1,2} \\
0 & 1 & 0 & a_{2,1} & a_{2,2}\\
0 & 0 & 1 & a_{3,1} & a_{3,2}
\end{smallpmatrix}$ & 
Exactly one pair of $\begin{smallpmatrix}a_{i,1}\cdot a_{j,1}\\a_{i,2}\cdot a_{j,2}\end{smallpmatrix}$, $i\neq j$, is linearly dependent. & 
$\begin{smallpmatrix}
y_1 & y_4 & y_5\\
y_4 & y_2 & 0\\
y_5   & 0 & y_3
\end{smallpmatrix}$ \\

\cmidrule{2-4}
 & $\begin{smallpmatrix}
1 & 0 & 0  & a_{1,1} & a_{1,2} \\
0 & 1 & 0 & a_{2,1} & a_{2,2}\\
0 & 0 & 1 & a_{3,1} & a_{3,2}
\end{smallpmatrix}$ & 
All pairs of $\begin{smallpmatrix}a_{i,1}\cdot a_{j,1}\\a_{i,2}\cdot a_{j,2}\end{smallpmatrix}$, $i\neq j$, are linearly independent.  &
$\begin{smallpmatrix}
y_1 & y_4 & y_4+y_5 \\
y_4 & y_2 & y_5 \\
y_4+y_5 & y_5 & y_3
\end{smallpmatrix}$ \\

\midrule
$6$ & $\begin{smallpmatrix}
1 & 0 & 0  & a_{1,1} & a_{1,2} & a_{1,3}\\
0 & 1 & 0 & a_{2,1} & a_{2,2} & a_{2,3}\\
0 & 0 & 1 & a_{3,1} & a_{3,2} & a_{3,3}
\end{smallpmatrix}$ & 
None & 
$\begin{smallpmatrix}
y_1 & y_4 & y_6\\
y_4 & y_2 & y_5 \\
y_6 & y_5 & y_3
\end{smallpmatrix}$ \\

\bottomrule
\end{tabular}
\end{table}

\begin{proof}
Let $W\subseteq\KK^E$ be a configuration of rank $r_W=3$ with normalized configuration matrix $A$.
By \eqref{7} and Lemma~\ref{40}, we may assume that
\[
3=r_W\le r_W^2=\abs{E}\le{r_W+1\choose 2}=6.
\]
The cases where $r_W^2\in\set{3,6}$ are covered by Proposition~\ref{11}.

\smallskip

Suppose now that $r_W^2=4$.
Up to reordering rows and columns, $A$ then has the form  
\[
A=
\begin{pmatrix}
1 & 0 & 0 & a_1\\
0 & 1 & 0 & a_2\\
0 & 0 & 1 & a_3
\end{pmatrix},\quad
a_1,a_2,a_3\in\KK,\quad
a_1a_2\ne0,
\]
and hence
\[
Q_A=
\begin{pmatrix}
x_1+a_1^2x_4 & a_1a_2x_4 & a_1a_3 x_4 \\
a_1a_2x_4 & x_2+a_2^2x_4 & a_2a_3x_4\\
a_1a_3x_4 & a_2a_3x_4 & x_3+a_3^2x_4
\end{pmatrix}.
\]
If $a_3=0$, then we can write, in terms of suitable coordinates $y_1,y_2,y_3,y_4$, 
\begin{equation}\label{55}
Q_A=
\begin{pmatrix}
y_1 & y_4 & 0 \\
y_4 & y_2 & 0 \\
0 &  0  & y_3
\end{pmatrix},\quad
\psi_A=\det(Q_A)=(y_1y_2-y_4^2)y_3.
\end{equation}
On the other hand, if $a_3\neq 0$, then we can write
\[
Q_{\lambda,\mu}:=Q_A=
\begin{pmatrix}
y_1 & y_4 & \mu y_4 \\
 y_4 & y_2 & \lambda y_4\\
\mu y_4 & \lambda y_4 & y_3
\end{pmatrix},\quad
\lambda:=\frac{a_3}{a_1},\quad
\mu:=\frac{a_3}{a_2}.
\]
Applying the coordinate change $(y_1,y_2,y_3,y_4)\mapsto(\frac{y_1}{\lambda^2},\frac{y_2}{\mu^2},y_3,\frac{y_4}{\lambda\mu})$, yields
\[
Q'_{\lambda,\mu}:=
\begin{pmatrix}
\frac{y_1}{\lambda^2} & \frac{y_4}{\lambda\mu} & \frac{y_4}{\lambda} \\
\frac{y_4}{\lambda\mu} & \frac{y_2}{\mu^2} & \frac{y_4}{\mu}\\
\frac{y_4}{\lambda} & \frac{y_4}{\mu} & y_3
\end{pmatrix},
\]
and hence by extracting factors from the first and second row and column
\[
\det(Q_{\lambda,\mu})\simeq\lambda^2\mu^2\det(Q'_{\lambda,\mu})=\det(Q_{1,1}).
\]
In contrast to $\psi_A$ in \eqref{55}, this cubic is irreducible since $\M_W=U_{3,4}$ is connected (see \cite[Thm.~4.16]{DSW21}).
In particular, the cases $a_3=0$ and $a_3\neq 0$ belong to different equivalence classes.

\smallskip

Suppose now that $r_W^2=5$. 
Then $A$ has the form
\[
A=
\begin{pmatrix}
1 & 0 & 0 & a_{1,1} & a_{1,2} \\
0 & 1 & 0 & a_{2,1} & a_{2,2} \\
0 & 0 & 1 & a_{3,1} & a_{3,2}
\end{pmatrix}.
\]

First suppose that, after suitably reordering the rows and columns of $A$, $w^1\star w^2$ and $w^2\star w^3$ are linearly dependent, and hence $w^1\star w^2$ and $w^1\star w^3$ are linearly independent. 
In terms of suitable coordinates $y_1,\dots,y_5$, we can write
\[
Q_\lambda:=Q_A=
\begin{pmatrix}
y_1 & y_4 &  y_5 \\
y_4 & y_2 & \lambda y_4 \\
y_5 & \lambda y_4 & y_3
\end{pmatrix},\quad\lambda\in \KK.
\]
By symmetric row and column operations, 
\[
\det(Q_\lambda)=\det
\begin{pmatrix}
y_1 & y_4 & y_5-\lambda y_1 \\
y_4 & y_2 & 0\\
y_5-\lambda y_1 & 0 & y_3-2\lambda y_5 +\lambda^2 y_1
\end{pmatrix}
\simeq\det(Q_0).
\]
One computes that the ideal of submaximal minors of $Q_0$ equals
\begin{equation}\label{56}
I_2(Q_0)=\ideal{y_1y_2-y_4^2,y_3,y_5}\cap\ideal{y_1y_3-y_5^2,y_2,y_4}.
\end{equation}

Suppose now that all pairs of $w^i\star w^j$ with $i<j$, are linearly independent.
In terms of suitable coordinates, $y_1,\dots,y_5$, we can write
\[
Q_{\lambda,\mu}=\begin{pmatrix}
y_1 & y_4 & \lambda y_4+\mu y_5 \\
y_4 & y_2 & y_5 \\
\lambda y_4+\mu y_5 & y_5 & y_3
\end{pmatrix},\quad
\lambda,\mu\in\KK^*.
\]
Applying the coordinate change $(y_1,y_2,y_3,y_4)\mapsto(\mu^2y_1,y_2,\lambda^2y_3,\mu y_4,\lambda y_5)$, yields
\[
Q'_{\lambda,\mu}=
\begin{pmatrix}
\mu^2y_1 & \mu y_4 & \lambda\mu(y_4+y_5)\\
\mu y_4 & y_2 & \lambda y_5 \\
\lambda\mu(y_4+y_5) & \lambda y_5 & \lambda^2 y_3
\end{pmatrix},
\]
and hence by extracting factors from the first and last row and column
\[
\det(Q_{\lambda,\mu})\simeq\frac1{\lambda^2\mu^2}\det(Q'_{\lambda,\mu})=\det(Q_{1,1}).
\]
The linear independence of all pairs of $w^i\star w^j$ with $i<j$ implies that $\M_W=U_{3,5}$ which is $3$-connected (see \cite[Table 8.1]{Oxl11}).
In contrast to $I_2(Q_0)$ in \eqref{56}, $I_2(Q_{1,1})$ must be a prime ideal (see \cite[Thm.~4.37]{DSW21}).
In particular, the two cases with $r_W^2=5$ belong to different equivalence classes.
\end{proof}
 
\section{Infinite number of classes for rank $4$ matroids}\label{71}

For rank $4$ configurations there are infinitely many equivalence
classes of configuration polynomials. 
For simplicity we prove this over the rationals, so in this section we assume $\KK=\QQ$.

Consider the family of normalized configuration matrices
\[
A:=
\begin{pmatrix}
1 & 0 & 0 & 0 & 1 & 1\\
0 & 1 & 0 & 0 & a_1 & b_1\\
0 & 0 & 1 & 0 & a_2 & 0  \\
0 & 0 & 0 & 1 & 0   & b_2 
\end{pmatrix},
\]
depending on parameters $a_1,a_2,b_1,b_2\in\QQ$ where $a_1a_2b_1b_2\neq 0$. 
We will see that it gives rise to an infinite family of polynomials 
\[
\psi_m:=\det(Q_m),\quad Q_m:=\begin{pmatrix}
y_1  &  y_5+y_6   &  y_5   &  m y_6  \\
y_5+y_6 &  y_2  &  y_5    &  y_6  \\
y_5          &  y_5 & y_3      &  0  \\
my_6 & y_6  & 0  & y_4 
\end{pmatrix},\quad m:=\frac{a_1}{b_1}\in\QQ,
\]
which are pairwise inequivalent for $\abs{m}>1$.


\begin{lem}\label{54}
With the above notation, we have $\psi_A\simeq\psi_m$.
\end{lem}

\begin{proof}
The configuration form associated to $A$ is given by 
\[
Q_A=\begin{pmatrix}
x_1+x_5+x_6             & a_1 x_5+b_1 x_6 & a_2 x_5   &  b_2x_6   \\
a_1 x_5+b_1 x_6 &       x_2+a_1^2x_5+b_1^2x_6       & a_1a_2x_5 & b_1b_2 x_6 \\
a_2 x_5         & a_1 a_2 x_5     &     x_3+a_2^2 x_5   & 0          \\
b_2x_6          &  b_1b_2x_6     &    0      &  x_4+b_2^2x_6       
\end{pmatrix}.
\]
The coordinate changes 
\begin{align*}
(z_1,\dots,z_6)&:=\tup{x_1+x_5+x_6,x_2+a_1^2x_5+b_1^2x_6,x_3+a_2^2 x_5,x_4+b_2^2 x_6,a_1 x_5,b_1 x_6},\\
(y_1,\ldots,y_6)&:=\tup{z_1,\frac{z_2}{a_1^2},\frac{z_3}{a_2^2},\frac{z_4}{b_2^2},\frac{z_5}{a_1},\frac{z_6}{a_1}}
\end{align*}
turn $Q_A$ into
\begin{align*}
Q_A&=\begin{pmatrix}
z_1       &    z_5+z_6       & \frac{a_2}{a_1}z_5    &   \frac{b_2}{b_1} z_6 \\
z_5+z_6   &      z_2         &   a_2 z_5             & b_2 z_6               \\
\frac{a_2}{a_1} z_5 & a_2z_5 &    z_3                &  0                    \\
\frac{b_2}{b_1} z_6 & b_2z_6 &    0                  & z_4
\end{pmatrix}\\
&=
\begin{pmatrix}
y_1&a_1(y_5+y_6)&a_2y_5&\frac{a_1b_2}{b_1}y_6\\
a_1(y_5+y_6)&a_1^2y_2&a_1a_2y_5&a_1b_2y_6\\
a_2y_5&a_1a_2y_5&a_2^2y_3&0\\
\frac{a_1b_2}{b_1}y_6&a_1b_2y_6&0&b_2^2y_4
\end{pmatrix},
\end{align*}
so that $\det(Q_A)=a_1^2a_2^2b_2^2\det(Q_m)$ by extracting factors from the last three rows and columns.
\end{proof}


\begin{prp}\label{53}
For $m,m'\in\QQ^*$, $\psi_m\simeq\psi_{m'}$ if and only if $m=m'$ or $mm'=1$.
\end{prp}

\begin{proof}
By a \textsc{Singular} computation, the primary decomposition of the ideal of submaximal minors of $Q_m$ reads
\[
I_2(Q_m)=P_{m,1}\cap P_{m,2}\cap P_{m,3}
\]
where 
\begin{align*}
P_{m,1}&=\langle y_1+my_2-(m+1)y_5-(m+1)y_6,\\
&y_2y_4-y_4y_5-y_4y_6+(m-1)y_6^2, my_2y_3-y_3y_5+(1-m)y_5^2-y_3y_6\rangle\\
P_{m,2}&=\ideal{y_6,y_4,y_1y_2y_3-y_5^2(y_1+y_2+y_3-2y_5)}\\
P_{m,3}&=\ideal{y_5,y_3,y_1y_2y_4-y_6^2(y_1+m^2y_2+y_4-2my_6)}
\end{align*}
Fix $m,m'\in\KK^*$ with $\psi_m\simeq\psi_{m'}$. 
Then there is an $\ell\in\GL_6(\KK)$ such that 
\[
\set{\ell^*(P_{m,i})\mid i\in\set{1,2,3}}=\set{\ell^*(P_{m',i})\mid i\in\set{1,2,3}}.
\]
Let us assume first that
\begin{equation}\label{19}
\ell^*(P_{m,1})=P_{m',1},\quad 
\ell^*(P_{m,2})=P_{m',2},\quad
\ell^*(P_{m,3})=P_{m',3}.
\end{equation}
Then $\ell^*$ stabilizes the vector spaces $\ideal{y_3,y_5}$ and $\ideal{y_4,y_6}$ and hence
\begin{align*}
\ell^*(y_3)& =\ell_{3,3}y_3+\ell_{3,5}y_5, & \ell^*(y_4)& =\ell_{4,4}y_4+\ell_{4,6}y_6, \\
\ell^*(y_5)&=\ell_{5,3}y_3+\ell_{5,5,}y_5, & \ell^*(y_6)&=\ell_{6,4}y_4+\ell_{6,6,}y_6.
\end{align*}
with non-vanishing determinants 
\begin{equation}\label{57}
\ell_{1,1}\ell_{2,2}-\ell_{1,2}\ell_{2,1}\ne 0,\quad 
\ell_{3,3}\ell_{5,5}-\ell_{3,5}\ell_{5,3}\ne 0,\quad 
\ell_{4,4}\ell_{6,6}-\ell_{4,6}\ell_{6,4}\ne 0.
\end{equation}
In degree $3$ the second equality in \eqref{19} yields
\begin{multline}\label{20}
(\ell_{3,3}y_3+\ell_{3,5}y_5)\sum_{i=1}^6\ell_{1,i}y_i\sum_{j=1}^6\ell_{2,j}y_j\\-(\ell_{5,3}y_3+\ell_{5,5}y_5)^2\tup{\sum_{i=1}^6(\ell_{1,i}+\ell_{2,i})y_i+(\ell_{3,3}-2\ell_{5,3})y_3+(\ell_{3,5}-2\ell_{5,5})y_5}\\
\equiv\lambda(y_1y_2y_3-y_5^2(y_1+y_2+y_3-2y_5))\mod\ideal{y_4,y_6},\quad\lambda\in\KK^*.
\end{multline}
By comparing coefficients of $y_1y_2y_5$ in \eqref{20}, we find $(\ell_{1,1}\ell_{2,2}+\ell_{1,2}\ell_{2,1})\ell_{3,5}=0$ which forces $\ell_{3,5}=0$ by \eqref{57}.
Comparing next the coefficients of the monomials 
\[
y_1^2,\quad y_2^2,\quad y_1y_5^2,\quad y_2y_5^2,
\]
in \eqref{20} we then obtain
\begin{align}\label{21}
\ell_{1,1}\ell_{2,1}&=0, & \ell_{1,2}\ell_{2,2}&=0,\\
\nonumber -\ell_{5,5}^2(\ell_{1,1}+\ell_{2,1})&=-\lambda, & -\ell_{5,5}^2(\ell_{1,2}+\ell_{2,2})&=-\lambda,
\end{align}
which yields
\begin{equation}\label{22}
\ell_{1,1}+\ell_{2,1}=\ell_{1,2}+\ell_{2,2}.
\end{equation}
In degree $1$ the first equality in \eqref{19} yields
\begin{multline}
\label{23}
\sum_{i=1}^6 \tup{(\ell_{1,i}+m\ell_{2,i})y_i} -(m+1)(\ell_{5,3}y_3+\ell_{5,5} y_5)-(m+1)(\ell_{6,4} y_4+\ell_{6,6} y_6)= \\
\mu \tup{y_1+m' y_2-(m'+1)y_5-(m'+1)y_6}.
\end{multline}
Comparing coefficients of $y_1$ and $y_2$ we find
\begin{align}
\label{24}
\ell_{1,1}+m\ell_{2,1}&=\mu, & \ell_{1,2}+m\ell_{2,2} & =m'\mu.
\end{align}
By equation~\eqref{21}, $\ell_{1,i}$ or $\ell_{2,i}$ must be zero for $i=1,2$.
Thus, we consider the following cases: 
\begin{asparaitem}

\item If $\ell_{1,1}=\ell_{1,2}=0$, then $\ell_{2,1}=\frac{\mu}{m}$ and $\ell_{2,2}=\frac{m'\mu}{m}$ by \eqref{24}, hence $\frac{\mu}{m}=\frac{m'\mu}{m}$ by \eqref{22}, so $m'=1$.

\item If $\ell_{1,1}=\ell_{2,2}=0$, then $\ell_{2,1}=\frac{\mu}{m}$ and $\ell_{1,2}={m'\mu}$ by \eqref{24}, hence $\frac{\mu}{m}={m'\mu}$ by \eqref{22}, so $m'=\frac{1}{m}$. 

\item If $\ell_{2,1}=\ell_{1,2}=0$, then $\ell_{1,1}={\mu}$ and $\ell_{2,2}=\frac{m'\mu}{m}$ by \eqref{24}, hence ${\mu}=\frac{m'\mu}{m}$ by \eqref{22}, so $m'=m$.

\item If $\ell_{2,1}=\ell_{2,2}=0$, then $\ell_{1,1}={\mu}$ and $\ell_{1,2}={m'\mu}$ by \eqref{24}, hence ${\mu}={m'\mu}$ by \eqref{22}, so $m'=1$.

\end{asparaitem}
A similar discussion applies, with the same consequences, to the case where
\begin{align*}
\ell(P_{m,1})&=P_{m',1}, & \ell(P_{m,2})&=P_{m',3}, & \ell(P_{m,3})&=P_{m',2}.
\end{align*}
In conclusion and by exchanging $\ell$ by $\ell^{-1}$, we find 
\[
m'\in\set{1,m,\frac{1}{m}},\quad m\in\set{1,m',\frac{1}{m'}}.
\]
Unless $m'=m$, we have $m'=\frac{1}{m}=\frac{b_1}{a_1}$.
In terms of the coordinates from the proof of Lemma~\ref{54}, we can write
\begin{align*}
\psi_A&=a_2^2b_2^2
\det\begin{pmatrix}
z_1       &    z_5+z_6       & \frac{z_5}{a_1}    &   \frac{z_6}{b_1} \\
z_5+z_6   &      z_2         &   z_5             & z_6               \\
\frac{z_5}{a_1} & z_5 & \frac{z_3}{a_2^2}  &  0                    \\
\frac{z_6}{b_1} & z_6 &    0                  & \frac{z_4}{b_2^2}
\end{pmatrix}\\
&\simeq
\det\begin{pmatrix}
z_1       &    z_5+z_6       & \frac{z_5}{a_1}    &   \frac{z_6}{b_1} \\
z_5+z_6   &      z_2         &   z_5             & z_6               \\
\frac{z_5}{a_1} & z_5 & z_3  &  0                    \\
\frac{z_6}{b_1} & z_6 &    0                  & z_4
\end{pmatrix}
\end{align*}
One can see that the morphism that leaves $z_1,z_2$ fixed, and interchanges the pairs $z_3\leftrightarrow z_4, z_5\leftrightarrow z_6, a_1\leftrightarrow b_1$ transforms this final matrix into a conjugate matrix. 
However, by Lemma \ref{54} the determinants of these two matrices are equivalent to $\psi_m$ and $\psi_{1/m}$ respectively, where $m=\frac{a_1}{b_1}$.
It follows that $\psi_m$ and $\psi_{1/m}$ are equivalent. 
\end{proof}


\begin{cor}\label{14}
For every $k\in\NN$, we have $|\Psi_{4+k}^{6+k}/_\simeq|=\infty$ over $\KK=\QQ$.
\end{cor}

\begin{proof}
Applying the construction from Remark~\ref{12} yields configurations $W$ with $r_W=4+k$ and $r_W^2=6+k$ which give rise to the infinite family of polynomials $\psi_{m,k}=\psi_m\cdot y_{7}\cdots y_{7+k}$, contact equivalent to elements of $\Psi_{4+k}^{6+k}$.
For reasons of degree, $\psi_{m,k}\simeq\psi_{m',k}$ is equivalent to $\psi_m\simeq\psi_{m'}$, so the claim follows from Proposition~\ref{53}.
\end{proof}

\printbibliography
\end{document}